\documentclass[
reqno]{amsart}
\usepackage{latexsym,amsmath,amssymb,amscd}
\usepackage[all]{xy}
\usepackage{enumerate}
\usepackage{hyperref}

\def\today{\ifcase \month \or
	January \or February \or March \or April \or
	May \or June \or July \or August \or
	September \or October \or November \or December \fi
	\space\number\day , \number\year}

\makeatletter
\newcommand\@dotsep{4.5}
\def\@tocline#1#2#3#4#5#6#7{\relax
	\ifnum #1>\c@tocdepth 
	\else
	\par \addpenalty\@secpenalty\addvspace{#2}%
	\begingroup \hyphenpenalty\@M
	\@ifempty{#4}{%
		\@tempdima\csname r@tocindent\number#1\endcsname\relax
	}{%
		\@tempdima#4\relax
	}%
	\parindent\z@ \leftskip#3\relax \advance\leftskip\@tempdima\relax
	\rightskip\@pnumwidth plus1em \parfillskip-\@pnumwidth
	#5\leavevmode\hskip-\@tempdima #6\relax
	\leaders\hbox{$\m@th
		\mkern \@dotsep mu\hbox{.}\mkern \@dotsep mu$}\hfill
	\hbox to\@pnumwidth{\@tocpagenum{#7}}\par
	\nobreak
	\endgroup
	\fi}
\makeatother

\begin{document}
	
	\makeatletter
	\@addtoreset{figure}{section}
	\def\thefigure{\thesection.\@arabic\c@figure}
	\def\fps@figure{h,t}
	\@addtoreset{table}{bsection}
	
	\def\thetable{\thesection.\@arabic\c@table}
	\def\fps@table{h, t}
	\@addtoreset{equation}{section}
	\def\theequation{
		\arabic{equation}}
	\makeatother
	
	\newcommand{\bfi}{\bfseries\itshape}
	\newtheorem{theorem}{Theorem}
	\newtheorem{corollary}[theorem]{Corollary}
	\newtheorem{criterion}[theorem]{Criterion}
	\newtheorem{definition}[theorem]{Definition}
	\newtheorem{example}[theorem]{Example}
	\newtheorem{lemma}[theorem]{Lemma}
	\newtheorem{notation}[theorem]{Notation}
	\newtheorem{problem}[theorem]{Problem}
	\newtheorem{proposition}[theorem]{Proposition}
	\newtheorem{remark}[theorem]{Remark}
	\numberwithin{theorem}{section}
	\numberwithin{equation}{section}

	\newcommand{\todo}[1]{\vspace{5 mm}\par \noindent
		\framebox{\begin{minipage}[c]{0.85 \textwidth}
				\tt #1 \end{minipage}}\vspace{5 mm}\par}


 \renewcommand{\1}{{\bf 1}}

 \newcommand{\hotimes}{\widehat\otimes}

 \newcommand{\Ad}{{\rm Ad}}
 \newcommand{\ad}{{\rm ad}}
 \newcommand{\Alt}{{\rm Alt}\,}
 \newcommand{\Ci}{{\mathcal C}^\infty}
 \newcommand{\comp}{\circ}
 \newcommand{\wt}{\widetilde}

 \newcommand{\ph}{\text{\bf P}}
 \newcommand{\conv}{{\rm conv}}
 \newcommand{\de}{{\rm d}}
 \newcommand{\ee}{{\rm e}}
 \newcommand{\ev}{{\rm ev}}
 \newcommand{\fimes}{\mathop{\times}\limits}
 \newcommand{\id}{{\rm id}}
 \newcommand{\ie}{{\rm i}}
 \newcommand{\End}{{\rm End}\,}
 \newcommand{\Gr}{{\rm Gr}}
 \newcommand{\GL}{{\rm GL}}
 \newcommand{\Hilb}{{\bf Hilb}\,}
 \newcommand{\Hom}{{\rm Hom}}
 \renewcommand{\Im}{{\rm Im}}
 \newcommand{\Ker}{{\rm Ker}\,}
 \newcommand{\Lie}{\textbf{L}}
 \newcommand{\lf}{{\rm l}}
 \newcommand{\Loc}{{\rm Loc}\,}
 \newcommand{\pr}{{\rm pr}}
 \newcommand{\Ran}{{\rm Ran}\,}
 \renewcommand{\Re}{{\rm Re}}
 \newcommand{\supp}{{\rm supp}\,}

 \newcommand{\Cb}{{\mathcal C}_b}
 \newcommand{\UCb}{{\mathcal U}{\mathcal C}_b}
 \newcommand{\LUCb}{{\mathcal L}{\mathcal U}{\mathcal C}_b}
 \newcommand{\RUCb}{{\mathcal R}{\mathcal U}{\mathcal C}_b}

 \newcommand{\Tr}{{\rm Tr}\,}
 \newcommand{\Tran}{\textbf{Trans}}

 \newcommand{\CC}{{\mathbb C}}
 \newcommand{\NN}{{\mathbb N}}
 \newcommand{\RR}{{\mathbb R}}
 \newcommand{\TT}{{\mathbb T}}
 \newcommand{\ZZ}{{\mathbb Z}}

 \newcommand{\G}{{\rm G}}
 \newcommand{\U}{{\rm U}}
 \newcommand{\Gl}{{\rm GL}}
 \newcommand{\SL}{{\rm SL}}
 \newcommand{\SU}{{\rm SU}}
 \newcommand{\VB}{{\rm VB}}

 \newcommand{\Ac}{{\mathcal A}}
 \newcommand{\Bc}{{\mathcal B}}
 \newcommand{\Cc}{{\mathcal C}}
 \newcommand{\Dc}{{\mathcal D}}
 \newcommand{\Ec}{{\mathcal E}}
 \newcommand{\Fc}{{\mathcal F}}
 \newcommand{\Gc}{{\mathcal G}}
 \newcommand{\Hc}{{\mathcal H}}
 \newcommand{\Kc}{{\mathcal K}}
 \newcommand{\Nc}{{\mathcal N}}
 \newcommand{\Oc}{{\mathcal O}}
 \newcommand{\Pc}{{\mathcal P}}
 \newcommand{\Qc}{{\mathcal Q}}
 \newcommand{\Rc}{{\mathcal R}}
 \newcommand{\Sc}{{\mathcal S}}
 \newcommand{\Tc}{{\mathcal T}}
 \newcommand{\Uc}{{\mathcal U}}
 \newcommand{\Vc}{{\mathcal V}}
 \newcommand{\Wc}{{\mathcal W}}
 \newcommand{\Xc}{{\mathcal X}}
 \newcommand{\Yc}{{\mathcal Y}}
 \newcommand{\Zc}{{\mathcal Z}}
 \newcommand{\Ag}{{\mathfrak A}}
 \renewcommand{\gg}{{\mathfrak g}}
 \newcommand{\hg}{{\mathfrak h}}
 \newcommand{\mg}{{\mathfrak m}}
 \newcommand{\nng}{{\mathfrak n}}
 \newcommand{\pg}{{\mathfrak p}}
 \newcommand{\Gg}{{\mathfrak g}}
 \newcommand{\Lg}{{\mathfrak L}}
 \newcommand{\Mg}{{\mathfrak M}}
 \newcommand{\Sg}{{\mathfrak S}}
 \newcommand{\Ug}{{\mathfrak u}}
 \newcommand{\zg}{{\mathfrak z}}

\markboth{}{}

\makeatletter
\title
[Representation theory of pro-Lie groups]{Coadjoint orbits in representation theory\\ of pro-Lie groups}
\author{Daniel Belti\c t\u a}
\address{Institute of Mathematics ``Simion Stoilow'' of the Romanian Academy,
	P.O. Box 1-764, Bucharest, Romania}
\email{Daniel.Beltita@imar.ro, beltita@gmail.com}
\author{Amel Zergane}
\address{Higher Institute of Applied Sciences and Technology of Sousse, Mathematical Physics Laboratory, Special Functions and Applications,
	City Ibn Khaldoun 4003, Sousse, Tunisia}
\email{amel.zergane@yahoo.fr}
\thanks{The work of the first-named author was supported by a grant of the Romanian National Authority for Scientific Research and
	Innovation, CNCS--UEFISCDI, project number PN-II-RU-TE-2014-4-0370.}
\date{19 September 2017
}
\makeatother


\begin{abstract}
We present a one-to-one correspondence between equivalence classes of unitary irreducible representations and
coadjoint orbits
for a class of pro-Lie groups
including all connected locally compact nilpotent groups and
arbitrary infinite direct products of nilpotent Lie groups.
\end{abstract}

\maketitle

\section{Introduction}

In this paper we sketch an approach to unitary representation theory for a class of projective limits of Lie groups, in the spirit of the method of coadjoint orbits from representation theory of Lie groups.
(See \cite{BZ17} for more details.)
The importance of this method stems from the fact that the groups under consideration here are not locally compact in general, hence they may not have a Haar measure, and therefore it is not possible to model their representation theory in the usual way, using Banach algebras or $C^*$-algebras.

By way of motivation, we discuss a simple example
(cf. \cite[Ex. 4.10]{BZ17}), which shows that
the usual $C^*$-algebraic approach to group representation theory does not work for topological groups which are not
locally compact.
Let $G=(\RR^\NN,+)$ be the abelian group which is the underlying additive group of the vector space of all sequences of real numbers.
Since the linear dual space $(\RR^\NN)^*=\RR^{(\NN)}$ is the vector space of all finitely supported sequences of real numbers,
it easily follows
that there exists
a bijection $\Psi_G\colon\widehat{G}\to\RR^{(\NN)}$
(compare also Corollary~\ref{O4_cor2}).
Specifically, for every $\lambda=(\lambda_j)_{j\in\NN}\in\RR^{(\NN)}$,
$\Psi_G^{-1}(\lambda)\in\widehat{G}$ is the equivalence class of the 1-dimensional representation
$$\chi_\lambda\colon G\to \U(1),\quad
\chi_\lambda((x_j)_{j\in\NN}):=\exp(\ie \sum_{j\in\NN}\lambda_jx_j)$$
where $\U(1):=\{z\in\CC\mid\vert z\vert=1\}$.
However, as the vector space $\RR^{(\NN)}$ is infinite dimensional, it is not locally compact,
hence it is not homeomorphic to the spectrum of any $C^*$-algebra.
Consequently, the irreducible representation theory of $G$ cannot be exhaustively described via any $C^*$-algebra.

\section{Preliminaries}

 \subsection*{Lie theory}
We use upper case Roman letters to denote Lie groups, and their corresponding lower case Gothic letters to denote the Lie algebras.
We will also use the notation $\Lie$ for the Lie functor which associates to each Lie group its Lie algebra,
hence for any Lie group $G$ one has $\Lie(G)=\gg$.
We denote the exponential map of a Lie group $G$ by $\exp_G\colon\gg\to G$, and if this map is bijective, then we denote its inverse by $\log_G\colon\gg\to G$.
For any morphism of Lie groups $q\colon G\to H$, its corresponding morphism of Lie algebras is denoted by $\Lie(q)\colon\gg\to\hg$,
hence one has the commutative diagram
$$\begin{CD}
\gg @>{\Lie(q)}>> \hg \\
@V{\exp_G}VV @VV{\exp_H}V \\
G @>{q}>> H
\end{CD} $$
The coadjoint action of a Lie group is denoted by $\Ad_G^*\colon G\times\gg^*\to\gg^*$, and its corresponding set of coadjoint orbits is denoted by $\gg^*/G$ or $\Lie(G)^*/G$.
If $q\colon G\to H$ is a surjective morphism of Lie groups,
then one has a map $\Lie(q)^*: \hg^*\to\gg^*$
such that for every coadjoint $H$-orbit $\Oc\in\hg^*/H$
its image $\Lie(q)^*(\Oc)$ is a coadjoint $G$-orbit,
and one thus obtains a map
$$\Lie(q)^*_{\Ad^*}\colon\hg^*/H\to\gg^*/G,\quad \Oc\mapsto\Lie(q)^*(\Oc).$$

\subsection*{Representation theory}

For any topological group $G$ we denote by $\widehat{G}$ its unitary dual, that is, its set of unitary equivalence classes $[\pi]$ of unitary irreducible representations $\pi\colon G\to\Bc(\Hc)$.
If $q\colon G\to H$ is a continuous surjective morphism of topological groups, then we define $$\widehat{q}\colon\widehat{H}\to\widehat{G}, \quad [\pi]\mapsto [\pi\circ q].$$

\begin{proposition}\label{O3a}
	Let $G$ be any connected nilpotent Lie group with its universal covering $p\colon\widetilde{G}\to G$,
	and denote $\Gamma:=\Ker p\subseteq \widetilde{G}$.
	We define
	$$\gg^*_{\ZZ}:=\{\xi\in\gg^*\mid (\xi\circ\Lie(p)\circ \log_{\widetilde{G}})(\Gamma)\subseteq \ZZ\}.$$
	Then the folowing assertions hold:
	\begin{enumerate}
	\item The set $\Gamma$ is a discrete subgroup of the center of~$\widetilde{G}$.
	\item The set $\gg^*_{\ZZ}$ is invariant to the coadjoint action of $G$.
	\item
	There exists an injective correspondence $\Psi_G\colon\widehat{G}\to\gg^*/G$, whose image is exactly
	the set of all coadjoint $G$-orbits contained in $\gg^*_{\ZZ}$, such that
	if $H$ is any other connected nilpotent Lie group with a surjective morphism of Lie groups $q\colon G\to H$, then one has the commutative diagram
	$$\begin{CD}
	\widehat{G} @>{\Psi_G}>> \gg^*/G  \\
	@A{\widehat{q}}AA @AA{\Lie(q)^*_{\Ad^*}}A \\
	\widehat{H} @>{\Psi_H}>> \hg^*/H
	\end{CD}$$
\end{enumerate}
\end{proposition}

\begin{proof}
	See \cite[Prop. A.3]{BZ17}.
\end{proof}

\section{Pro-Lie groups and their Lie algebras}

The main results that we will give below (see Theorem~\ref{O4} and its corollaries) are applicable to pro-Lie groups and are stated in terms of Lie algebras and coadjoint orbits of these groups.
Therefore we discuss these notions in this section.
Our general reference for pro-Lie groups is the monograph \cite{HM07}, and we also refer to the paper \cite{HN09} for the relation between pro-Lie groups and infinite-dimensional Lie groups.

Any topological group in this paper is assumed to be Hausdorff by definition.
A Cauchy net in a topological group $G$ is a net $\{g_j\}_{j\in J}$ in $G$ with the property that for every neighborhood $V$ of $\1\in G$ there exists $j_V\in J$ such that for all $i,k\in J$ with $i\ge j_V$ and $k\ge j_V$ one has $g_ig_k^{-1}\in V$.
A topological group $G$ is called \emph{complete} if every Cauchy net in $G$ is convergent.
Every locally compact group is complete by \cite[Rem. 1.31]{HM07}.

For any topological group $G$ we denote by $\Nc(G)$ the set of its \emph{co-Lie subgroups}, that is, the closed normal subgroups $N\subseteq G$ for which $G/N$ is a finite-dimensional Lie group.
 We say that $G$ is a \emph{pro-Lie group} if it is complete and for every neighborhood $V$ of $\1\in G$ there exists $N\in\Nc(G)$ with $N\subseteq V$ (cf. \cite[Def. 3.25]{HM07}).
 If this is the case, then $\Nc(G)$ is closed under finite intersections, hence it is a filter basis (cf. \cite[page 148]{HM07}).

 Pro-Lie groups can be equivalently defined as the limits of projective systems of Lie groups, by \cite[Th. 3.39]{HM07}.

\begin{definition}
	\normalfont
	For any pro-Lie group $G$, its set of continuous 1-parameter subgroups 
	$$\Lie(G):=\{X\in\Cc(\RR,G)\mid
	(\forall t,s\in\RR)\quad X(t+s)=X(t)X(s)\}$$
	is endowed with its topology of uniform convergence on the compact subsets of~$\RR$.
	 Then
	the topological space $\Lie(G)$ has the structure of a locally convex Lie algebra over~$\RR$,
	whose scalar multiplication, vector addition and bracket
	satisfy the following conditions for all
	$t,s\in\RR$ and $X_1,X_2\in\Lie(G)$:
	\allowdisplaybreaks
	\begin{eqnarray}
	\nonumber
	(t\cdot X_1)(s) &=& X_1(ts);  \\
	\nonumber
	(X_1+X_2)(t) &=& \lim\limits_{n\to\infty}(X_1(t/n)X_2(t/n))^n;\\
	\nonumber
	[X_1,X_2](t^2) &=& \lim\limits_{n\to\infty}(X_1(t/n)X_2(t/n)X_1(-t/n)X_2(-t/n))^{n^2},
	\end{eqnarray}
	where the convergence is uniform on the compact subsets of~$\RR$.
	(See for instance \cite[Ex. 2.7(4.)]{BB11}.)
One also has the dual vector space
$$\Lie(G)^*:=\{\xi\colon\Lie(G)\to\RR\mid\xi\text{ is linear and continuous}\}$$
endowed with its locally convex topology of pointwise convergence on~$\Lie(G)$.
The \emph{adjoint action} is 
$\Ad_G\colon G\times\Lie(G)\to\Lie(G)$, $(g,X)\mapsto\Ad_G(g)X:=g X(\cdot)g^{-1} $,
and this defines by duality the \emph{coadjoint action}
$$\Ad_G^*\colon G\times\Lie(G)^*\to\Lie(G)^*,\quad (g,\xi)\mapsto\Ad_G^*(g)\xi:=\xi\circ \Ad_G(g^{-1}). $$
We denote by $\Lie(G)^*/G$ the set of all coadjoint orbits, that is, the orbits of the above coadjoint action.
\end{definition}
In the following proposition we summarize a few basic properties of Lie algebras of connected locally compact groups.
A pro-Lie group $G$ is called \emph{pronilpotent} if for every $N\in\Nc(G)$ the finite-dimensional Lie group $G/N$ is nilpotent.
(See \cite[Def. 10.12]{HM07}.)

\begin{proposition}\label{locpro}
	If $G$ is a connected locally compact group, then the following assertions hold:
	\begin{enumerate}
		\item\label{locpro_1} $G$
	is a pro-Lie group and its Lie algebra $\Lie(G)$ is the direct product of a finite-dimensional Lie algebra,
	an abelian (possibly infinite-dimensional) Lie algebra, and a (possibly infinite) product of simple compact Lie algebras.
	\item The following conditions are equivalent:
	\begin{enumerate}
		\item\label{locpro_a} The group $G$ is pronilpotent.
		\item\label{locpro_b} The Lie algebra $\Lie(G)$ is the product of a finite-dimensional nilpotent Lie algebra and an abelian (possibly infinite-dim\-ension\-al) Lie algebra.
		\item\label{locpro_c} The Lie algebra $\Lie(G)$ is nilpotent (possibly infinite-dim\-ension\-al).
		\item\label{locpro_d} The group $G$ is nilpotent.
	\end{enumerate}
\end{enumerate}
\end{proposition}

\begin{proof}
	The first assertion follows by \cite[Th. 4]{Gl57} or \cite[Cor. 4.24]{La57}.
	See also \cite[Th. 2.1.2.2]{BCR81}.
	
	For the second assertion, we first recall from \cite[Th. 10.36 and Def. 7.42]{HM07} that the group $G$ is pronilpotent if and only if its Lie algebra $\Lie(G)$ is pronilpotent, that is, every finite-dimensional quotient algebra of $\Lie(G)$ is nilpotent.
	Therefore, in view of Assertion~\ref{locpro_1}, one has $$\eqref{locpro_a}\iff\eqref{locpro_b}\iff\eqref{locpro_c}.$$
	Moreover, one clearly has  $\eqref{locpro_d}\implies\eqref{locpro_a}$.
	
	We now prove $\eqref{locpro_b}\implies\eqref{locpro_d}$.
	To this end let $\pi_G\colon\widetilde{G}\to G$ be the universal morphism defined in \cite[Def. 4.20]{La57} and \cite[page 259]{HM07}.
	Then $\Lie(\pi_G)\colon\Lie(\widetilde{G})\to\Lie(G)$ is an isomorphism of Lie algebras and the image of $\pi_G$ is dense in $G$ by \cite[Th. 6.6 (i) and (iv)]{HM07}.
	It follows at once by condition \eqref{locpro_b} and \cite[Th. 4.23]{La57} that the group $\widetilde{G}$ is nilpotent.
	Then, as the image of $\pi_G$ is dense in $G$, we obtain~\eqref{locpro_d}, and this completes the proof.
\end{proof}

\section{Main results}

Theorem~\ref{O4} below provides an exhaustive description of the unitary dual of a class of topological groups that are not locally compact.
As we discussed in the introduction, unitary dual spaces of non-locally-compact groups in general cannot be described in terms of representation theory of $C^*$-algebras.

For the following definition we recall that if $X$ is an arbitrary nonempty set, then a \emph{filter basis} on $X$
is a nonempty set~$B$
whose elements are nonempty subsets of $X$ having the property that for any $X_1,X_2\in B$ there exists $X_0\in B$ with $X_0\subseteq X_1\cap X_2$.
If $X$ is moreover endowed with a topology, then one says that the filter basis $B$ \emph{converges} to a point $x_0\in X$ if for every neighborhood ~$V$ of $x_0$ there exists $X\in B$ with $X\subseteq V$.

\begin{example}
	\normalfont
	Here are some basic examples of filter bases. 
\begin{enumerate}
\item Every neighborhood basis at any point of a topological space is a filter basis converging to that point.
\item If $G$ is a group endowed with the discrete topology and $B$ is a set of subgroups
of~$G$ such that the trivial subgroup $G_0:=\{\1\}$
is an element of~$B$, then $B$ is a filter basis on $G$ converging to $\1\in G$ since for any $G_1,G_2\in B$ one has $G_0\subseteq G_1\cap G_2$ and on the other hand $G_0$ is contained in any neighborhood of $\1\in G$.
\item If $G$ is a topological group with the property that for every neighborhood $V$ of $\1\in G$ there exists a co-Lie subgroup $N\in\Nc(G)$ with $N\subseteq V$, then $\Nc(G)$ is a filter basis on $G$ converging to $\1\in G$ since in fact for every $N_1,N_2\in\Nc(G)$ one has $N_1\cap N_2\in\Nc(G)$.
(See \cite[page 148]{HM07}.)
In particular, this holds true for pro-Lie groups.
\end{enumerate}
\end{example}

\begin{definition}
\normalfont
An \emph{amenable filter basis} on a topological group~$G$
is a filter basis $\Nc\subseteq\Nc(G)$ converging to $\1\in G$ such that every topological group $N\in\Nc$ is amenable.
\end{definition}
	
\begin{example}
\normalfont
Here are two examples of amenable filter basis that are needed in
Corollaries \ref{O4_cor1}--\ref{O4_cor2}:
\begin{enumerate}
	\item If $G$ is a connected locally compact group,
	then $\Nc(G)$ is an amenable filter basis.
	In fact, every $N\in\Nc(G)$ is compact hence amenable,
	and on the other hand $\Nc(G)$ converges to $\1\in G$ by the theorem of Yamabe.
	(See for instance \cite[Th. 0.1.5]{BCR81}.)
	\item Let $\{G_j\}_{j\in J}$ be an infinite family of nilpotent Lie groups with their direct product topological group $G:=\prod\limits_{j\in J}G_j$.
	Denote by $\Nc$ the set of all subgroups of $G$ of the form $N_F:=\prod\limits_{j\in J}N_j$ associated to any finite subset $F\subseteq J$, with $N_j=\{\1\}\subseteq G_j$ if $j\in F$ and $N_j=G_j$ if $j\in J\setminus F$.
	It is clear that every $N_F$ of this form has the following properties: $N_F$ is a closed normal subgroup of $G$ that is isomorphic to $\prod\limits_{j\in J\setminus F}G_j$ hence $N_F$ is amenable by \cite[Prop. 3.8]{BZ17}, and moreover $G/N_F$ is isomorphic to $\prod\limits_{j\in F}G_j$, which is a Lie group since $F$ is a finite set, hence $N_F\in \Nc(G)$.
	For any finite subsets $F_1,F_2\subseteq J$ one clearly has
	$N_{F_1}\cap N_{F_2}=N_{F_1\cup F_2}$, where $F_1\cup F_2$ is again a finite subset of $J$, hence $\Nc$ is a filter basis on $G$.
	Moreover, by the definition of an infinite direct product of topologies, it follows that the filter basis $\Nc$ converges to $\1\in G$.
	Consequently, $\Nc$ is an amenable filter basis on~$G$.
\end{enumerate}
\end{example}
	
\begin{theorem}\label{O4}
	Let $G$ be a complete topological group with an amenable filter basis~$\Nc$ for which $G/N$ is a connected nilpotent Lie group for every $N\in\Nc$.
	Then there exists a well-defined bijective correspondence
	$$\Psi_G\colon \widehat{G}\to \Lie(G)^*/G,\quad [\pi]\mapsto\Oc^\pi$$
	between the equivalence classes of unitary irreducible representations of $G$ and the set of all coadjoint $G$-orbits contained in the $G$-invariant set
	$$\Lie(G)^*_{\ZZ}:=\{\xi\in\Lie(G)^*\mid (\exists N\in\Nc)(\exists \eta\in\Lie(G/N)^*_\ZZ)\quad \xi=\eta\circ\Lie(p_N)\}.$$
	Every unitary irreducible representation $\pi\colon G\to\Bc(\Hc)$ is thus associated to the coadjoint $G$-orbit $\Oc^\pi:=\Lie(p_N)^*(\Oc_0)\subseteq\Lie(G)^*_{\ZZ}$, where $N\in\Nc$ and $\Oc_0\subseteq \Lie(G/N)^*_{\ZZ}$
	is the coadjoint $(G/N)$-orbit associated with a unitary irreducible representation $\pi_0\colon G/N\to\Bc(\Hc)$
	satisfying $\pi_0\circ p_N=\pi$.
\end{theorem}

\begin{proof}
	See \cite[Th. 4.6]{BZ17}.
\end{proof}

In connection with the following corollary we note that the Lie algebras of connected locally compact nilpotent groups can be described as in Proposition~\ref{locpro}.

\begin{corollary}\label{O4_cor1}
	If $G$ is a connected locally compact nilpotent group,
	then there is a bijective correspondence
	$\Psi_G\colon \widehat{G}\to \Lie(G)^*/G$
	onto the set of all coadjoint $G$-orbits contained in
	a certain $G$-invariant subset $\Lie(G)^*_{\ZZ}\subseteq \Lie(G)$.
	For any filter basis $\Nc\subseteq\Nc(G)$
	converging to  the identity
	one has
	$$\Lie(G)^*_{\ZZ}:=\{\xi\in\Lie(G)^*\mid (\exists N\in\Nc)(\exists \eta\in\Lie(G/N)^*_\ZZ)\quad \xi=\eta\circ\Lie(p_N)\}.$$
\end{corollary}

\begin{proof}
		See \cite[Cor. 4.7]{BZ17}.
\end{proof}

We now draw a corollary of Theorem~\ref{O4} that applies to pro-Lie groups which are not locally compact.

\begin{corollary}\label{O4_cor2}
	If $\{G_j\}_{j\in J}$ is a family of connected nilpotent Lie groups,
	with their direct product topological group $G:=\prod\limits_{j\in J}G_j$,
	then there is a bijective correspondence
	$\Psi_G\colon \widehat{G}\to \Lie(G)^*/G$
	onto the set of all coadjoint $G$-orbits contained in the
	$G$-invariant subset  $\Lie(G)^*_{\ZZ}\subseteq \Lie(G)$.
	Here we define
	$$\Lie(G)^*_{\ZZ}:=\{\xi\in\Lie(G)^*\mid (\exists F\in\Fc)(\exists\eta\in\Lie(G_F)^*_{\ZZ})\quad \xi=\eta\circ\Lie(p_F)\}$$
	where $\Fc$ is the set of all finite subsets $F\subseteq J$,
	and for every $F\in\Fc$ we define $G_F:=\prod\limits_{j\in F}G_j$ and $p_F\colon G\to G_F$ is the natural projection.
\end{corollary}

\begin{proof}
	See \cite[Cor. 4.9]{BZ17}.
\end{proof}

\begin{remark}\label{18Sept2017}
\normalfont 
The amenability hypotheses of Theorem~\ref{O4} 
may actually be removed, using some results of \cite{Ne10}. 
\end{remark}



\subsection*{Acknowledgment} 
We are grateful to Karl-Hermann Neeb for Remark~\ref{18Sept2017}.

\end{document}